\documentclass[draft]{amsart}
\usepackage{enumitem}
\usepackage{amsrefs}
\usepackage{mathrsfs}

\newtheorem{thm}{Theorem}[section]
\newtheorem{cor}[thm]{Corollary}
\numberwithin{equation}{section}

\author{Wayne Grey}
\address{Department of Mathematics,
University of Western Ontario,
London, Canada}
\email{wgrey@uwo.ca}

\author{Gord Sinnamon}
\address{Department of Mathematics,
University of Western Ontario,
London, Canada}
\email{sinnamon@uwo.ca}
\thanks{Supported by the Natural Sciences and Engineering Research Council of Canada}


\begin{document}

\title{Product operators on mixed norm spaces}

\begin{abstract} Inequalities for product operators on mixed norm Lebesgue spaces and permuted mixed norm Lebesgue spaces are established. They depend only on inequalities for the factors and on the Lebesgue indices involved. Inequalities for the bivariate Laplace transform are given to illustrate the method. Also, an elementary proof is presented for an $n$-variable Young's inequality in mixed norm spaces.
\end{abstract}

\maketitle

\section{Introduction}

The techniques used to study embeddings of mixed norm spaces in \cite{thesis} and \cite{GS} can be extended to work with operators other than the identity. Here we begin this work by considering a select class of operators. Mixed norm spaces are spaces of multivariable functions in which the norm takes advantage of the product structure in the domain. They have a long informal history but were first named and formally studied by Benedek and Panzone in \cite{BP}. Permuted mixed norms only arise when studying more than a single space, since they occur when the order in which the factor norms are taken is different in different spaces. The importance of permuted mixed norms was noted in \cite{F} and \cite[\S XI.1]{KA}. They also appear in \cite{AK}, a study of the Hardy operator in mixed-norm Lebesgue spaces, and in the papers \cite{KZ} and \cite{AKZ}, in which partial integral operators are considered between mixed-norm Banach function spaces, with Lebegue and Orlicz norms as special cases. A systematic study of continuous inclusions between mixed-norm Lebesgue spaces was undertaken in the first author's thesis, \cite{thesis}. 

It will be convenient to introduce the two-variable mixed norm spaces needed in Section \ref{prod} first, postponing the introduction of $n$-variable spaces until Section \ref{conv}. Let $(T_1,\lambda_1)$ and $(T_2,\lambda_2)$ be $\sigma$-finite measure spaces and let $L_{\lambda_1\times\lambda_2}^0$ denote the collection of $(\lambda_1\times\lambda_2)$-measurable functions on $T_1\times T_2$. In particular, either or both of $\lambda_1$ and $\lambda_2$ may be ``counting'' measure. Thus, the definitions of this section and the results of the next apply to weighted or unweighted sums as special cases. 

Fix indices $p_1,p_2\in (0,\infty)$. For any $f\in L_{\lambda_1\times\lambda_2}^0$,
\[
\|f\|_{L^{(p_1,p_2)}_{\lambda_1\times\lambda_2}}
=\left(\int_{T_2}
\left(\int_{T_1} |f(t_1,t_2)|^{p_1}\,d\lambda_1(t_1)\right)^{p_2/p_1}\,d\lambda_2(t_2)\right)^{1/p_2}.
\]
The first variable of the function $f$ is always in the $\lambda_1$ measure space, and the order of the indices and measures indicates which variable is the ``inner'' one. So for any $f\in L_{\lambda_1\times\lambda_2}^0$,
\[
\|f\|_{L^{(p_2,p_1)}_{\lambda_2\times\lambda_1}}
=\left(\int_{T_1}
\left(\int_{T_2} |f(t_1,t_2)|^{p_2}\,d\lambda_2(t_2)\right)^{p_1/p_2}\,d\lambda_1(t_1)\right)^{1/p_1}.
\]
Although these are genuine norms only when $p_1\ge1$ and $p_2\ge1$ we will refer to them as mixed norms even when some indices are less than 1. Also, the results we present will often extend to the case when one or both of the indices is infinite, indicating that the supremum norm is to be taken in that factor.

The product operators that we consider will be introduced in Section \ref{prod}. As mentioned above, we restrict our attention to the two-variable case. Although the extension from bivariate operators to multivariate operators may seem straightforward, the delicate arguments needed in the embedding case in \cite{thesis} and \cite{GS} and the advanced techniques introduced in \cite{G} put this extension beyond the scope of the present article.

Convolution operators have some product structure but not enough to make them product operators in general. Convolution is considered in Section \ref{conv}, where it leads to an elementary proof of a multivariate mixed norm Young's inequality. 

The Hardy operators considered in \cite{AK} are product operators (and also  convolution operators) with the added advantage that weighted norm inequalities for the factors have been characterized. The results may be compared with Theorem \ref{K}, below. Hardy operators are not the only operators with these useful properties, see for example, the Riemann-Liouville operators studied in \cite{St}.

We will make frequent use of Minkowski's (integral) inequality in the following mixed norm form: If $0<p_1\le p_2<\infty$, then,
\[
\|f\|_{L^{(p_1,p_2)}_{\lambda_1\times\lambda_2}}
\le \|f\|_{L^{(p_2,p_1)}_{\lambda_2\times\lambda_1}}, \quad f\in L_{\lambda_1\times\lambda_2}^+.
\]
For easy recognition, we will enclose the function in square brackets when applying Minkowski's inequality in integral estimates. It becomes,
\begin{multline*}
\left(\int_{T_2}\left(\int_{T_1} [f(t_1,t_2)]^{p_1}\,d\lambda_1(t_1)\right)^{p_2/p_1}
\,d\lambda_2(t_2)\right)^{1/p_2}\\
\le \left(\int_{T_1}\left(\int_{T_2} [f(t_1,t_2)]^{p_2}\,d\lambda_2(t_2)\right)^{p_1/p_2}
\,d\lambda_1(t_1)\right)^{1/p_1}.
\end{multline*}
\section{Product operators}\label{prod} Let $(T_1,\lambda_1)$, $(T_2,\lambda_2)$, $(X_1,\mu_1)$, and $(X_2,\mu_2)$ be $\sigma$-finite measure spaces.  Let $L_{\lambda_1\times\lambda_2}^+$ denote the collection of non-negative $(\lambda_1\times\lambda_2)$-measurable functions. 

An operator $K:L_{\lambda_1\times\lambda_2}^+\to L_{\mu_1\times\mu_2}^+$ will be called a {\it product operator} provided it can be expressed in the form,
\[
Kf(x_1,x_2)=\iint_{T_1\times T_2} k_1(x_1,t_1)k_2(x_2,t_2)f(t_1,t_2)\,d(\lambda_1\times\lambda_2)(t_1,t_2),
\]
where $k_j$ is a non-negative $(\mu_j\times\lambda_j)$-measurable function for $j=1,2$. In this case we define $K_j:L_{\lambda_j}^+\to L_{\mu_j}^+$ by
\[
K_jf(x)=\int_{T_j} k_j(x,t)f(t)\,d\lambda_j(t), \quad j=1,2.
\]
Note that, by Tonelli's theorem,
\begin{align*}
Kf(x_1,x_2)&=\int_{T_1} k_1(x_1,t_1)\int_{T_2} k_2(x_2,t_2)f(t_1,t_2)\,d\lambda_2(t_2)\,d\lambda_1(t_1)\\
&=\int_{T_2} k_2(x_2,t_2)\int_{T_1} k_1(x_1,t_1)f(t_1,t_2)\,d\lambda_1(t_1)\,d\lambda_2(t_2).
\end{align*}

Suppose $p_1$, $p_2$, $r_1$, and $r_2$ are positive and let $C_j$ be the least constant, finite or infinite, such that
\[
\|K_jf\|_{L^{r_j}_{\mu_j}}\le C_j\|f\|_{L^{p_j}_{\lambda_j}},\quad f\in L_{\lambda_j}^+.
\]

\begin{thm}\label{K}  Fix positive indices $p_1$, $p_2$, $r_1$, and $r_2$. Suppose $K$ is a product operator, with $k_j$, $K_j$, $C_j$ as above for $j=1,2$.\begin{enumerate}[leftmargin=1.85em, label=\rm{(\alph*)}]
\item\label{Ka} If $r_1\ge1$ then $K$ satisfies the mixed norm inequality,
\[
\|Kf\|_{L^{(r_1,r_2)}_{\mu_1\times\mu_2}}
\le C_1C_2\|f\|_{L^{(p_1,p_2)}_{\lambda_1\times\lambda_2}}, \quad f\in L_{\lambda_1\times\lambda_2}^+.
\]
\item\label{Kb} If $r_1\ge1$, $r_2\ge 1$ and $\min(p_1,r_1)\le\max(p_2,r_2)$ then $K$ satisfies the permuted mixed norm inequality,
\[
\|Kf\|_{L^{(r_1,r_2)}_{\mu_1\times\mu_2}}
\le C_1C_2\|f\|_{L^{(p_2,p_1)}_{\lambda_2\times\lambda_1}}, \quad f\in L_{\lambda_1\times\lambda_2}^+.
\]
\end{enumerate}
\end{thm}
\begin{proof} In a slight abuse of notation we write,
\begin{align*}
K_1f(x_1,t_2)&=\int_{T_1} k_1(x_1,t_1)f(t_1,t_2)\,d\lambda_1(t_1)
\quad\text{and}\\
K_2f(t_1,x_2)&=\int_{T_2} k_2(x_2,t_2)f(t_1,t_2)\,d\lambda_2(t_2).
\end{align*}

To prove part \ref{Ka}, use the hypothesis $r_1\ge1$ to apply Minkowski's  inequality and then invoke the definitions of $C_1$ and $C_2$. This yields,
\begin{align*}
&\left(
\int_{X_2}\left(
\int_{X_1} Kf(x_1,x_2)^{r_1}
\,d\mu_1(x_1)\right)^{r_2/r_1}
\,d\mu_2(x_2)\right)^{1/r_2}\\
&=\left(
\int_{X_2}\left(
\int_{X_1} \left(
\int_{T_2} \Big[k_2(x_2,t_2)K_1f(x_1,t_2)\Big]
\,d\lambda_2(t_2)\right)^{r_1}\!\!
\,d\mu_1(x_1)\right)^{r_2/r_1}\!\!
\,d\mu_2(x_2)\right)^{1/r_2}\\
&\le\left(
\int_{X_2}\left(
\int_{T_2} \left(
\int_{X_1}  \Big[k_2(x_2,t_2)K_1f(x_1,t_2)\Big]^{r_1}
\,d\mu_1(x_1)\right)^{1/r_1}\!\!
\,d\lambda_2(t_2)\right)^{r_2}\!\!
\,d\mu_2(x_2)\right)^{1/r_2}\\
&=\left(
\int_{X_2}\left(
\int_{T_2} k_2(x_2,t_2)\left(
\int_{X_1}  K_1f(x_1,t_2)^{r_1}
\,d\mu_1(x_1)\right)^{1/r_1}\!\!
\,d\lambda_2(t_2)\right)^{r_2}\!\!
\,d\mu_2(x_2)\right)^{1/r_2}\\
&\le C_1\left(
\int_{X_2}\left(
\int_{T_2} k_2(x_2,t_2)\left(
\int_{T_1}  f(t_1,t_2)^{p_1}
\,d\lambda_1(t_1)\right)^{1/p_1}\!\!
\,d\lambda_2(t_2)\right)^{r_2}\!\!
\,d\mu_2(x_2)\right)^{1/r_2}\\
&\le C_1C_2\left(
\int_{T_2}\left(
\int_{T_1}f(t_1,t_2)^{p_1}
\,d\lambda_1(t_1)\right)^{p_2/p_1}\!\!
\,d\lambda_2(t_2)\right)^{1/p_2}.
\end{align*}

Part \ref{Kb} will be done in four cases: They arise from the observation that the condition $\min(p_1,r_1)\le\max(p_2,r_2)$ is satisfied if and only if one of, $p_1\le p_2$, $r_1\le r_2$, $p_1\le r_2$, or $r_1\le p_2$ holds. If $p_1\le p_2$ then part \ref{Ka}, followed by Minkowski's inequality, yields
\[
\|Kf\|_{L^{(r_1,r_2)}_{\mu_1\times\mu_2}}
\le C_1C_2\|f\|_{L^{(p_1,p_2)}_{\lambda_1\times\lambda_2}}\le C_1C_2\|f\|_{L^{(p_2,p_1)}_{\lambda_2\times\lambda_1}}, \quad f\in L_{\lambda_1\times\lambda_2}^+.
\]
If $r_1\le r_2$ then we begin by using Minkowski's inequality, and follow with part \ref{Ka} to get
\[
\|Kf\|_{L^{(r_1,r_2)}_{\mu_1\times\mu_2}}
\le\|Kf\|_{L^{(r_2,r_1)}_{\mu_2\times\mu_1}}
\le C_2C_1\|f\|_{L^{(p_2,p_1)}_{\lambda_2\times\lambda_1}}, \quad f\in L_{\lambda_1\times\lambda_2}^+.
\]
If $p_1\le r_2$ then we apply the definition of $C_1$ followed by Minkowski's inequality and then the definition of $C_2$ to get
\begin{align*}
&\left(
\int_{X_2}\left(
\int_{X_1} Kf(x_1,x_2)^{r_1}
\,d\mu_1(x_1)\right)^{r_2/r_1}
\,d\mu_2(x_2)\right)^{1/r_2}\\
&=\left(
\int_{X_2}\left(
\int_{X_1} \left(
\int_{T_1} k_1(x_1,t_1)K_2f(t_1,x_2)
\,d\lambda_1(t_1)\right)^{r_1}
\,d\mu_1(x_1)\right)^{r_2/r_1}
\,d\mu_2(x_2)\right)^{1/r_2}\\
&\le C_1\left(
\int_{X_2}\left(
\int_{T_1}  \Big[K_2f(t_1,x_2)\Big]^{p_1}
\,d\lambda_1(t_1)\right)^{r_2/p_1}
\,d\mu_2(x_2)\right)^{1/r_2}\\
&\le C_1\left(
\int_{T_1}\left(
\int_{X_2}  \Big[K_2f(t_1,x_2)\Big]^{r_2}
\,d\mu_2(x_2)\right)^{p_1/r_2}
\,d\lambda_1(t_1)\right)^{1/p_1}\\
&\le C_1C_2\left(
\int_{T_1}\left(
\int_{T_2}  f(t_1,t_2)^{p_2}
\,d\lambda_2(t_2)\right)^{p_1/p_2}
\,d\lambda_1(t_1)\right)^{1/p_1}.
\end{align*}
If $r_1\le p_2$ the process is somewhat lengthy, using the definitions of $C_1$ and $C_2$ as well as three applications of Minkowski's inequality. First, apply Minkowski's inequality with $r_1\ge1$, followed by the definition of $C_2$ to get 
\begin{align*}
&\left(
\int_{X_2}\left(
\int_{X_1} Kf(x_1,x_2)^{r_1}
\,d\mu_1(x_1)\right)^{r_2/r_1}
\,d\mu_2(x_2)\right)^{1/r_2}\\
&=\left(
\int_{X_2}\left(
\int_{X_1} \left(
\int_{T_2} \Big[k_2(x_2,t_2)K_1f(x_1,t_2)\Big]
\,d\lambda_2(t_2)\right)^{r_1}\!\!
\,d\mu_1(x_1)\right)^{r_2/r_1}\!\!
\,d\mu_2(x_2)\right)^{1/r_2}\\
&\le\left(
\int_{X_2}\left(
\int_{T_2} \left(
\int_{X_1} \Big[k_2(x_2,t_2) K_1f(x_1,t_2)\Big]^{r_1}
\,d\mu_1(x_1)\right)^{1/r_1}\!\!
\,d\lambda_2(t_2)\right)^{r_2}\!\!
\,d\mu_2(x_2)\right)^{1/r_2}\\
&=\left(
\int_{X_2}\left(
\int_{T_2} k_2(x_2,t_2)\left(
\int_{X_1}  K_1f(x_1,t_2)^{r_1}
\,d\mu_1(x_1)\right)^{1/r_1}\!\!
\,d\lambda_2(t_2)\right)^{r_2}\!\!
\,d\mu_2(x_2)\right)^{1/r_2}\\
&\le C_2\left(
\int_{T_2}\left(
\int_{X_1}  \Big[K_1f(x_1,t_2)\Big]^{r_1}
\,d\mu_1(x_1)\right)^{p_2/r_1}
\,d\lambda_2(t_2)\right)^{1/p_2}.
\end{align*}
Now Minkowski's inequality with $p_2\ge r_1$ shows that the last expression is no larger than
\[
C_2\left(
\int_{X_1}\left(
\int_{T_2}  \Big[K_1f(x_1,t_2)\Big]^{p_2}
\,d\lambda_2(t_2)\right)^{r_1/p_2}
\,d\mu_1(x_1)\right)^{1/r_1}.
\]
After expanding $K_1f(x_1,t_2)$ in this expression we may apply Minkowski's inequality with $p_2\ge1$ to get
\begin{align*}
&C_2\left(
\int_{X_1}\left(
\int_{T_2}  \left(
\int_{T_1} \Big[k_1(x_1,t_1)f(t_1,t_2)\Big]
\,d\lambda_1(t_1)\right)^{p_2}
\,d\lambda_2(t_2)\right)^{r_1/p_2}
\,d\mu_1(x_1)\right)^{1/r_1}\\
&\le C_2\left(
\int_{X_1}\left(
\int_{T_1}  \left(
\int_{T_2} \Big[k_1(x_1,t_1)f(t_1,t_2)\Big]^{p_2}
\,d\lambda_2(t_2)\right)^{1/p_2}\!\!
\,d\lambda_1(t_1)\right)^{r_1}\!\!
\,d\mu_1(x_1)\right)^{1/r_1}\\
&= C_2\left(
\int_{X_1}\left(
\int_{T_1} k_1(x_1,t_1) \left(
\int_{T_2} f(t_1,t_2)^{p_2}
\,d\lambda_2(t_2)\right)^{1/p_2}\!\!
\,d\lambda_1(t_1)\right)^{r_1}\!\!
\,d\mu_1(x_1)\right)^{1/r_1}\\
&\le C_1C_2\left(
\int_{T_1}\left(
\int_{T_2}  f(t_1,t_2)^{p_2}
\,d\lambda_2(t_2)\right)^{p_1/p_2}
\,d\lambda_1(t_1)\right)^{1/p_1},
\end{align*}
where the last inequality uses the definition of $C_1$. These estimates complete the proof of the fourth and last case.
\end{proof}

The index conditions in Theorem \ref{K}\ref{Kb} are somewhat stronger than  needed. It is enough to assume that one (or more) of the following four conditions holds: $p_1\le r_2$; $1\le r_1\le p_2$; $1\le r_1$ and $p_1\le p_2$; or $1\le r_2$ and $r_1\le r_2$. The statement of the theorem remains valid when some indices are infinite, provided the index conditions are met, but the proofs would have to be modified to accommodate occurrences of the supremum norm. 

The inequalities of Theorem \ref{K} are stated for non-negative functions only but it is routine to extend the operator $K$ to all functions for which the right hand side is finite, in a way that preserves the norm inequalities. 

Also, the results of Theorem \ref{K} may be seen to hold for a more general class of positive operators. Instead of supposing that the factors $K_1$ and $K_2$ are integral operators with non-negative kernels, it would suffice to assume that they map positive functions to positive functions and possess formal adjoints. See \cite{Sposops}*{Lemma 2.4} for properties of such operators and \cite{HS}*{Section 4} for connections with positive integral operators. One advantage of such an extension is that the identity operator is not an integral operator (in general) but it does have a formal adjoint. 

\begin{cor}\label{Laplace} Suppose $p_1,p_2\in (1,2]$. The bivariate Laplace  transform $\mathscr L_2$, defined by
\[
\mathscr L_2f(x_1,x_2)=\int_0^\infty\int_0^\infty e^{-x_1t_1-x_2t_2}f(t_1,t_2)\,dt_1\,dt_2,\quad x_1>0,\,x_2>0,
\]
satisfies the permuted mixed norm inequality
\[
\|\mathscr L_2f\|_{L^{(p_1',p_2')}}
\le C(p_1)C(p_2)\|f\|_{L^{(p_2,p_1)}}.
\]
Here $1/p_1+1/p_1' =1/p_2+1/p_2'= 1$, and $C(p)$ is the norm of the single-variable Laplace transform as a map from $L^p$ to $L^{p'}$.
\end{cor}
\begin{proof} The bivariate Laplace transform is a product operator, in the above sense, whose factors are both the single-variable Laplace transform, $\mathscr L$, given by
\[
\mathscr Lf(x)=\int_0^\infty e^{-xt}f(t)\,dt,\quad x>0.
\]
It is well known that $\mathscr L$ is bounded as a map from $L^p$ to $L^{p'}$, i.e. $C(p)$ is finite, when $1<p\le 2$. Since $p_1'\ge1$, $p_2'\ge1$, and $\min(p_1,p_1')=p_1\le2\le p_2'=\max(p_2,p_2')$ the second statement of Theorem \ref{K} gives the conclusion.
\end{proof}

Note that if $1<p_2<p_1\le 2$, the permuted mixed norm inequality given above can be used to refine the unpermuted one given by Theorem \ref{K}\ref{Ka} in two different ways. Corollary \ref{Laplace} and Minkowski's inequality give
\[
\|\mathscr L_2f\|_{L^{(p_1',p_2')}}
\le C(p_1)C(p_2)\|f\|_{L^{(p_2,p_1)}}
\le C(p_1)C(p_2)\|f\|_{L^{(p_1,p_2)}},
\]
and also,
\[
\|\mathscr L_2f\|_{L^{(p_1',p_2')}}
\le\|\mathscr L_2f\|_{L^{(p_2',p_1')}}
\le C(p_1)C(p_2)\|f\|_{L^{(p_1,p_2)}}.
\]

Using Theorem \ref{K} it is easy to generate additional permuted mixed norm inequalities, by beginning with known single-variable inequalities. See, for example,  \cite{A}*{Corollary 1} for more single-variable inequalities involving the Laplace transform.

\section{Young's inequality}\label{conv}

Fix a positive integer $n$ and let $L^+$ be the collection of non-negative Lebesgue measurable functions on $\mathbb R^n$. For $P=(p_1,\dots,p_n)\in [1,\infty)^n$ define
\[
\|f\|_{L^P}=\left(\int_{\mathbb R}\dots\left(\int_{\mathbb R}
\left(\int_{\mathbb R} |f(t_1,\dots,t_n)|^{p_1}\,dt_1\right)^{p_2/p_1}\,dt_2\right)^{p_3/p_2}\dots\,dt_n\right)^{1/p_n}.
\]

The convolution of two (real-valued) functions on $\mathbb R^n$ is defined by,
\[
f*g(x)=\int_{\mathbb R^n} f(x-t)g(t)\,dt,
\]
whenever the integral exists. 

For a fixed function $g$ the map $f\mapsto f*g$ has a kind of product structure. But, even in the two-variable case, it is not a product operator of the sort considered in the previous section unless $g$ factors as $g(t_1,t_2)=g_1(t_1)g_2(t_2)$. This will keep us from establishing permuted mixed norm inequalities. Nevertheless, exploiting the existing product structure provides an elementary proof of a $n$-variable mixed norm Young's inequality.

Recall the single-variable Young's inequality: If $p,q,r\in [1,\infty]$ satisfy $1/p+1/q=1/r+1$, $f\in L^p$ and $g\in L^q$, then $f*g$ is well defined and $\|f*g\|_{L^r}\le\|f\|_{L^p}\|g\|_{L^q}$.

\begin{thm} Suppose $P=(p_1,\dots,p_n)$, $Q=(q_1,\dots,q_n)$, and $R=(r_1,\dots r_n)$ satisfy, 
\[
\frac1{p_j}+\frac1{q_j}=\frac1{r_j}+1,\quad j=1,\dots,n.
\]
If $f\in L^P$ and $g\in L^Q$ then $f*g$ is well defined, and
\[
\|f*g\|_{L^R}\le\|f\|_{L^P}\|g\|_{L^Q}.
\]
\end{thm}
\begin{proof} We begin by proving the inequality when $f$ and $g$ are non-negative, so that $f*g$ is well-defined as a function taking values in $[0,\infty]$.

As the first step in a recursive argument, let $f_1=f$ and $g_1=g$. Define $\hat x$ and $\hat t$ so that $(x_1,x_2,\dots,x_n)=(x_1,\hat x)$ and $(t_1,t_2,\dots,t_n)=(t_1,\hat t)$. Also, let $d\hat t$ denote $dt_2,\dots dt_n$. Then Young's inequality gives, for each $\hat x$ and $\hat t$,
\begin{multline*}
\left(\int_{\mathbb R}\left(\int_{\mathbb R} f_1(x_1-t_1,\hat x-\hat t)g_1(t_1,\hat t)\,dt_1\right)^{r_1}\,dx_1\right)^{1/r_1}\\
\le\left(\int_{\mathbb R} f_1(y_1,\hat x-\hat t)^{p_1}\,dy_1\right)^{1/p_1}
\left(\int_{\mathbb R} g_1(t_1,\hat t)^{q_1}\,dt_1\right)^{1/q_1}.
\end{multline*}
So, applying Minkowski's inequality, we get
\begin{align*}
&\left(\int_{\mathbb R} f_1*g_1(x_1,\hat x)^{r_1}\,dx_1\right)^{1/r_1}\\
&=\left(\int_{\mathbb R}\left(\int_{\mathbb R^{n-1}}\left[\int_{\mathbb R} f_1(x_1-t_1,\hat x-\hat t)g_1(t_1,\hat t)\,dt_1\right]\,d\hat t\right)^{r_1}\,dx_1\right)^{1/r_1}\\
&\le\int_{\mathbb R^{n-1}}\left(\int_{\mathbb R}\left[\int_{\mathbb R} f_1(x_1-t_1,\hat x-\hat t)g_1(t_1,\hat t)\,dt_1\right]^{r_1}\,dx_1\right)^{1/r_1}d\hat t\\
&\le\int_{\mathbb R^{n-1}}
\left(\int_{\mathbb R} f_1(y_1,\hat x-\hat t)^{p_1}\,dy_1\right)^{1/p_1}
\left(\int_{\mathbb R} g_1(t_1,\hat t)^{q_1}\,dt_1\right)^{1/q_1}\,d\hat t.
\end{align*}
Now define $f_2,g_2:\mathbb R^{n-1}\to [0,\infty]$ by
\[
f_2(y_2,\dots,y_n)=\left(\int_{\mathbb R} f_1(y_1,y_2,\dots,y_n)^{p_1}\,dy_1\right)^{1/p_1}
\]
and
\[
g_2(t_2,\dots,t_n)=\left(\int_{\mathbb R} g_1(t_1,t_2,\dots,t_n)^{q_1}\,dt_1\right)^{1/q_1}.
\]
We have just shown that,
\begin{equation}\label{1}
\left(\int_{\mathbb R} f_1*g_1(x_1,x_2,\dots,x_n)^{r_1}\,dx_1\right)^{1/r_1}
\le f_2*g_2(x_2,\dots,x_n).
\end{equation}
Applying the above argument to $f_2$ and $g_2$ gives
\begin{equation}\label{2}
\left(\int_{\mathbb R} f_2*g_2(x_2,x_3,\dots,x_n)^{r_2}\,dx_2\right)^{1/r_2}
\le f_3*g_3(x_3,\dots,x_n),
\end{equation}
where
\[
f_3(y_3,\dots,y_n)=\left(\int_{\mathbb R} f_2(y_2,y_3,\dots,y_n)^{p_2}\,dy_2\right)^{1/p_2}
\]
and
\[
g_3(t_3,\dots,t_n)=\left(\int_{\mathbb R} g_2(t_2,t_3,\dots,t_n)^{q_2}\,dt_2\right)^{1/q_2}.
\]
We continue in this way until reaching 
\begin{equation}\label{n}
\left(\int_{\mathbb R} f_{n-1}*g_{n-1}(x_{n-1},x_n)^{r_{n-1}}\,dx_{n-1}\right)^{1/r_{n-1}}
\le f_n*g_n(x_n),
\end{equation}
where
\[
f_n(y_n)=\left(\int_{\mathbb R} f_{n-1}(y_{n-1},y_n)^{p_{n-1}}\,dy_{n-1}\right)^{1/p_{n-1}}
\]
and
\[
g_n(t_n)=\left(\int_{\mathbb R} g_{n-1}(t_{n-1},t_n)^{q_{n-1}}\,dt_{n-1}\right)^{1/q_{n-1}}.
\]
Then Young's inequality gives,
\begin{equation}\label{Yn}
\left(\int_{\mathbb R} f_n*g_n(x_n)^{r_n}\,dx_n\right)^{1/r_n}
\le \left(\int_{\mathbb R} f_n(t_n)^{p_n}\,dt_n\right)^{1/p_n}
\left(\int_{\mathbb R} g_n(t_n)^{q_n}\,dt_n\right)^{1/q_n}.
\end{equation}
Recursively applying the definitions of $f_j$ and $g_j$ shows that the right hand side of (\ref{Yn}) is just
\[
\|f\|_{L^P}\|g\|_{L^Q}.
\]
The convolution estimates (\ref1),(\ref2),...,(\ref{n}) concatenate to give a lower bound for the left hand side of (\ref{Yn}) and we have
\[
\|f*g\|_{L^R}\le\|f\|_{L^P}\|g\|_{L^Q}.
\]
Now we drop the assumption of positivity on $f\in L^P$ and $g\in L^Q$. For bounded, integrable functions the convolution exists and is finite everywhere. And it is routine to express $f$ and $g$ as pointwise limits of bounded, integrable functions $f_k$ and $g_k$, respectively, satisfying $|f_k|\le |f|$ and $|g_k|\le |g|$. Since $|f|\in L^P$ and $|g|\in L^Q$ we have shown $|f|*|g|\in L^R$ and hence $|f|*|g|(x)<\infty$ almost everywhere. So for almost every $x$ the dominated convergence theorem proves that $f*g(x)$ exists. Also,
\[
\|f*g\|_{L^R}\le\||f|*|g|\|_{L^R}\le\||f|\|_{L^P}\||g|\|_{L^Q}=\|f\|_{L^P}\|g\|_{L^Q}.
\]
This completes the proof.
\end{proof}
Once again, the statement of the theorem remains valid when some indices are infinite, but the proof would have to be modified to accommodate occurrences of the supremum norm.

The same straightforward procedure may be applied to prove a mixed norm Young's inequality over any finite product of locally compact unimodular groups.

\end{document}